\documentclass[letterpaper,10pt]{amsart}
\usepackage{amsmath}
\usepackage{calc}
\usepackage[dvipsnames]{xcolor}
\usepackage{accents}
\usepackage[all]{xy}                      %
  
\CompileMatrices                            

\UseTips                                    

\input xypic
\usepackage[bookmarks=true]{hyperref}       

\usepackage{amssymb,latexsym,amsmath,amscd}
\usepackage{xspace}
\usepackage{color}
\usepackage{graphicx}
\usepackage{dsfont}

\reversemarginpar

\theoremstyle{plain}
\newtheorem{theorem}{Theorem}[]
\newtheorem*{theorem*}{Theorem}

\newtheorem{corollary}[theorem]{Corollary}
\newtheorem{lemma}[theorem]{Lemma}

\theoremstyle{definition}

\newtheorem{example}[theorem]{Example}

\newcommand{\enm}[1]{\ensuremath{#1}}          %
\newcommand{\op}[1]{\operatorname{#1}}
\newcommand{\cal}[1]{\mathcal{#1}}

\newcommand{\PP}{\enm{\mathbb{P}}}

\newcommand{\Mm}{\enm{\cal{M}}}

\newcommand{\Oo}{\enm{\cal{O}}}

\newcommand{\Uu}{\enm{\cal{U}}}

\renewcommand{\phi}{\varphi}
\renewcommand{\theta}{\vartheta}
\renewcommand{\epsilon}{\varepsilon}

\newcommand{\Hom}{\op{Hom}}

\newcommand{\End}{\op{End}}

\newcommand{\ce}{\mathrel{\mathop:}=}

\renewcommand{\to}[1][]{\xrightarrow{\ #1\ }}

\newcommand{\old}[1]{}

\date{}

\begin{document}

\date{}
\author{Edoardo Ballico and Elizabeth Gasparim}  
\address{E.B. - Univ. of Trento, Italy; edoardo.ballico@unitn.it;
E.G. - Research Fellow of the Geometry and AI Group,  the Federal Univ. of Technology, Paraná, Brazil; etgasparim@gmail.com}
\title{On the meagerness of the set of irregular bundles on Hopf surfaces}

\maketitle

\begin{abstract} 
Let $\mathcal M_{r,c}$ denote  the moduli space  of stable bundles
with rank $r$ and second Chern class $c>0$ on 
a Hopf surface.
We prove that the subset of  $\mathcal M_{r,c}$  formed by irregular bundles is meager. \\

\end{abstract}


 \section*{Motivation and main result}
 
 Our interest in vector bundles on Hopf surfaces  is motivated 
 by open questions about the moduli spaces of instantons on $S^3\times S^1$ 
 described by  Edward Witten. As a corollary of our result, we  answer in the affirmative 
 the question of connectedness (posed in \cite[Sec.\thinspace 2.1.4] {W})
  of the moduli of $SU(2)$  instantons.

Let $\pi\colon X\to \mathbb P^1$ be a classical Hopf surface.
Hence $X$ is elliptically fibered and  diffeomorphic  to $S^3\times S^1$.
 Let $\mathcal M_{r,c}$ denote  the moduli space  of stable bundles
with rank $r$ and second Chern class $c>0$ on $X$. 
Our main result here is:

\begin{theorem*}
Let $\Uu$ be a connected component of $\Mm_{r,c}$. Then the set of all $E\in \Uu$ which are not regular, i.e. either with jumps or irregular, is meager in $\Uu$
(and hence the set of all $E\in \Uu$ which are regular is dense in every open subset of $\Uu$.
\end{theorem*}

 \section*{A caveat and some basic concepts}

 Before proving the main theorem, it is worth pointing out why it can not be argued 
 simply by analogy with the algebraic case. Using the concepts of 
 filtrable and unstable bundles, we give  $2$ illustrative examples of new phenomena 
 appearing in the non-K\"ahler setup. 
 
 \begin{example} It is well know that bundles on projective surfaces are filtrable;
 e.g. if $E$ is a rank $2$ bundle on a a smooth projective surface then 
 then it may be written as an extension 
 $$0 \to L \to E \to F\to 0$$ 
 of a rank $1$ sheaf by a line bundle, see \cite[Ch.\thinspace 2, Prop.\thinspace 5]{Fr}.
 
 However, \cite[Thm.\thinspace 3.9]{BG2} proves that filtrable bundles on a classical
  Hopf surface $X$ must have jumps.
 Therefore, the dense open set of $\mathcal M_{r,c}$  formed by stable bundles without jumps 
 consists solely of nonfiltrable bundles. 
 \end{example}

 \begin{example} On \cite[App.\thinspace A]{BGRS} it is 
 shown that the moduli stack of rank $2$ vector bundles on a projective variety  of general type 
 has entire components formed by uniquely of unstable bundles, with some such components
  having dimension strictly larger than that of the moduli space of stable bundles.
  With this in mind, a priori one might have expected that there  also existed  large components
of the moduli of bundles on the Hopf surface which are  formed only  of irregular bundles.
 It is the purpose of this note to prove this is not the case.
 \end{example}
 
  We now establish some notation and basic definitions. 
Fix a Gauduchon metric on the classical Hopf $X$. Stability will always be with respect to this metric.
For each pair of integers $r\ge 2$ and $c>0$, let $\Mm_{r,c}$ denote the moduli space of all 
stable rank $r$ vector bundles $E$ on $X$ such that $c_2(E)=c$ and $\det(E)\cong \Oo_X$.
 It is known that $\Mm_{r,c}$ is smooth and of dimension $2rc$.

 We now recall a couple of elementary  concepts from topology. Let $M$ be a topological space.
 The subspace $A \subset M$ is called nowhere dense in $M$ if $\mbox{int}_M(\overline{A}) = \emptyset$, 
 that is, if the closure of $A$  contains no open subset of $M$. 
 The  set $A$ is called meager if it is a countable union $A = \cup A_n$ of nowhere dense  subsets $A_n \subset M$.

Let $\pi\colon X\to \PP^1$ be the elliptic fibration of the Hopf surface. For each $x\in \PP^1$ set $T_x\ce \pi^{-1}(x)$. 
Denote by $T$ be the elliptic curve such that $\pi$ is a principal $T$-bundle, hence $T_x\cong T$ for all $x\in \PP^1$. Take $E\in \Mm_{r,c}$. We say that $E$ has a jump if there is $x\in \PP^1$ such that $E_{|T_x}$ is not semistable. 

Let $F$ be a rank $r$ semistable vector bundle on $T$ such that $\det(F)\cong \Oo_T$.
 Recall  from \cite[Def.\thinspace 1.12]{FMW} that $F$ is said to be regular 
if its indecomposable factors are all different i.e. associated to pairwise distinct  elements of $\mathrm{Pic}(T)$. Otherwise $F$ is said to be irregular on $T$. Regular bundles present good behaviour inside the moduli spaces, because 
they have the smallest possible amount of automorphisms. 
A bundles $E$ on $X$ is said to be irregular if there exists an $x \in \mathbb P^1$ such that $E\vert_{T_x}$ is irregular. 
Otherwise, $E$ is said to be regular on $X$.

  \section*{The meagerness of the irregular locus}
 
\begin{lemma}\label{i3}
Let $F$ be a rank $r\ge 2$ vector bundle on an elliptic curve $T$ such that $\det(F)\cong \Oo_T$. By Atiyah's classification $F$ is a direct sum of semistable indecomposable ones, $h^0(T,\End(F)) \ge r$ and equality holds if  $F$ is regular. If $F$ is not regular, then $h^0(T,\End(F)) \ge r+2$.
\end{lemma}

\begin{proof}
For the first statement is proved in \cite[Lem.\thinspace 1.13]{FMW}. Now assume that $F$ is not regular. In particular, $F$ is decomposable. Write $F=F_1\oplus \cdots \oplus F_s$ with each $F_i$ indecomposable. Let $r_i$ be the rank of $F_i$. Set $d_i:= \deg(F_i)$. 

First, assume that $F$ is not semistable. Since $\det(F)$ has degree $0$, there is $i\in \{1,\dots,s\}$ such that $d_i>0$. Since $d_1+\cdots +d_s=0$, there is an index $j$ such that $d_j<0$. Let $e$ (resp.~$f$) be the number of factors of $F$ with positive (resp.~negative) degree. We have $e>0$, $f>0$ and $e+f\le s$. 
Permuting the factors of $F$ we may assume $d_a\ge d_b$ for all $a\ge b$. For all $a, b\in \{1,\dots,s\}$ the vector bundle $\Hom(F_b,F_a)$ is semistable of rank $r_ar_b$ and of degree $d_ar_b -d_br_a$. Fix $a,b\in \{1,\dots ,s\}$ such that $d_a>d_b$. Since $\Hom(F_b,F_a)$ is semistable of rank $r_ar_b$ and of degree $d_ar_b -d_br_a>0$, we have $h^0(T,\Hom(F_b,F_a)) =d_ar_b-d_br_a$. Write $F =F[-1]\oplus F[0]\oplus F[1]$ with $F[0]$ the direct sum of the degree $0$ factors of $F$, if any, and $F[-1]$ the direct sum of the negative degree factors and $F[1]$ the sum of the positive degree factors of $F$. Set $a_i:= \mathrm{rank}(F[i])$. We have $a_0+a_1+a_{-1}=r$, $a_0\ge 0$, $a_1\ge 1$ and $a_2\ge 1$. Since each factor of $F[1]$ has degree $\ge 1$ and each factor of $F[-1]$ has degree $\le -1$, we have $h^0(T,\Hom(F[-1],F[1]) \ge a_1+a_{-1}$. We also have $h^0(T,\Hom(F[-1],F[0]) \ge a_0$. Since $H^0(T,\Hom(F,F))$ contains also $H^0(T,\Hom(F[-1],F[-1])$ and $H^0(T,\Hom(F[1],F[1])$ and these vector spaces are non-zero (containing the multiplication by constants), we get $h^0(T,\Hom(F,F)) \ge a_1+a_0+a_2+2=r+2$.

Now assume that $F$ is semistable. Hence (with the notation of \cite[\S 1]{FMW}) $F_i\cong I_{r_i}(\lambda_i)$ for some $\lambda_i\in T^\ast$. By \cite[Lemma 1.13]{FMW} we have $h^0(T,\End(F_i)) =r_i$. Hence $h^0(T,\End(F)) =r +\sum_{i\ne j} h^0(T,\Hom(F_i,F_j))$ (in the sum we take all $i<j$ and all $i>j$).
Since $F$ is irregular, there are $i, j$ such that $i\ne j$ and $\lambda_i =\lambda_j$. Hence $h^0(T,\Hom(F_i,F_j))\ge 1$ and $h^0(T,\Hom(F_j,F_i))\ge 1$.
We conclude that $h^0(T,\End(F)) \ge r+2$.
\end{proof}

\begin{theorem}\label{i1}
Let $\Uu$ be a connected component of $\Mm_{r,c}$. Then the set of all $E\in \Uu$ which are not regular, i.e. either with jumps or irregular, is meager in $\Uu$
(and hence the set of all $E\in \Uu$ which are regular is dense in every open subset of $\Uu$.
\end{theorem}

\begin{proof}
Let $A$ be the set of all $E\in \Uu$ which are not regular, i.e. such that there is $x\in \PP^1$ for which $E_{|F_x}$ is not regular. For each $x\in \PP^1$, let  $A_x$ denote the set of all $E\in \Uu$ such that $E_{|T_x}$ is not regular.
Since $\dim \PP^1 =1$, it is sufficient to prove that each $A_x$ is locally, around each of its points,
 a closed analytic subset of $\Uu$ of codimension at least
$2$. Fix $x\in \PP^1$ and $E\in A_x$. Set $F:= E_{|T_x}$. Consider the exact sequence
\begin{equation}\label{eqd1}
0\to \End(E)(-T_x)\to \End(E)\to \End(F)\to 0 \text{.}
\end{equation}
Duality gives $$h^2(X,\End(E)(-T_x)) =h^0(X,\End(E)(+T_x)\otimes \omega_X) = h^0(X,\End(E)(-T_x)).$$ 
Since $E$ is simple, each element of $H^0(\End(E))$ is induced by the multiplication by a constant. Therefore, $h^0(X,\End(E)(-T_x))=0$.

Thus, the linear map $\rho\colon H^1(X,\End(E))\to H^1(T_x,\End(F))$ is surjective, and an open neighborhood
 of $0\in H^1(X,\End(E))$ is biholomorphic to an open neighborhood of $E$ in $\Mm_{r,c}$. An open neighborhood of $0$ in $H^1(T_x,\End(F))$ is a versal deformation space for $F$. Since $\rho$ is surjective, Lemma \ref{i3} gives that 
$A_x$ has locally at $E$ at least codimension $2$ in $\Uu$.
\end{proof}

The following corollary follows implicitly from  \cite{BH} and is discussed in further detail in \cite{G}, 
where it is proved under the assumption that each component of the moduli $\mathcal M_{2,c}$
contains a stable bundle. Our Theorem \ref{i1} shows that such an assumption is a truism.

\begin{corollary}\label{i2}
Take $r=2$. The complex manifold $\Mm_{2,c}$ is connected.
\end{corollary}

\begin{proof}
By Theorem \ref{i1} it is sufficient to prove that the set of all regular $E\in \Mm_{c,2}$ it is connected. 
Recall that to each  $E\in \Mm_{2,c}$ there corresponds a  spectral curve $C_E$, which 
 when $E$ is regular  is a smooth and connected curve of genus $2n-1$ \cite[Thm.\thinspace 1]{BG}.
To each  $E\in \Mm_{2,c}$ there also corresponds a graph $G_E$, and each graph is associated to a unique spectral curve.
 
 The set of all graphs associated to at least one regular bundle is an irreducible Zariski open subset of $\PP^{2n+1}$.
  Furthermore,  for each spectral curve $C_E$ associated to a regular bundle, all bundles $F$ 
  having  $C_E=C_F$ are regular \cite[Lem.\thinspace 1]{BG}.
  Moreover,  the set of such bundles (i.e. the fibre of the graph map $G^{-1}(E)$) is
   parametrized, by a smooth and connected projective variety, the Jacobian of the smooth genus $2n-1$ curve $C_E$.
\end{proof}

\paragraph{\bf Acknowledgements}  
 We are grateful to E. Witten for enlightening correspondence.
 E. Gasparim thanks Carlos Varea, director of mathematics of  Universidade Tecnológica 
 Federal do Paraná, Cornélio Procópio (Brazil), for  the kind offer of membership in the 
research group on Geometry and AI. 
 E. Ballico is a member of  GNSAGA of INdAM (Italy). 
 E. Gasparim is a senior associate of the Abdus Salam International 
 Centre for Theoretical Physics, Trieste (Italy) and thanks the hospitality of the ICTP
 where this note was completed.

\end{document}